\documentclass[11pt,reno]{amsart}  
\usepackage{amssymb}
\usepackage{amsmath}
\usepackage{enumitem}
\usepackage{mathrsfs}
\usepackage[english]{babel}
\usepackage[all]{xy}
\usepackage{hyperref}
\usepackage{marginnote}
\usepackage{comment}

\usepackage{stmaryrd}

 \usepackage{fullpage}
\usepackage[margin=1.30in]{geometry}

\RequirePackage{mathrsfs} 

\newtheorem{theorem}{Theorem}[section]

\newtheorem{conjecture}[theorem]{Conjecture}

\theoremstyle{definition}
\newtheorem*{ack}{Acknowledgements}
\newtheorem*{con}{Conventions}

\newtheorem{definition}[theorem]{Definition}

\numberwithin{equation}{section} \numberwithin{figure}{section}

\usepackage{color}

\definecolor{orange}{rgb}{1,0.5,0}

\title[Hilbert irreducibility   for varieties with a nef tangent bundle]{Hilbert irreducibility  for varieties with  a nef tangent bundle}

\author{Ariyan Javanpeykar}
\address{Ariyan Javanpeykar \\
Institut f\"{u}r Mathematik\\
Johannes Gutenberg-Universit\"{a}t Mainz\\
Staudingerweg 9, 55099 Mainz\\
Germany.}
\email{peykar@uni-mainz.de}



\begin{document}

 \begin{abstract}   
 We prove a   fibration property for varieties with Hilbert-type properties and give applications to  rational points on   varieties with nef tangent bundle. 
\end{abstract}

\maketitle

\thispagestyle{empty}

  \section{Introduction} 
 How many points should a variety with a dense set of rational points "really" have? Conjecturally,   a smooth projective variety over a number field with a dense set of rational points should have even more points than visible at first hand. This vague philosophy is made more precise by a conjecture of  Corvaja-Zannier relating density of rational points to   Hilbert-type properties. Let us be more precise.
 
A normal integral variety $X$ over a field $k$  of characteristic zero satisfies the \emph{Hilbert property over $k$} (see \cite[\S 3]{Serre}) if, for every finite collection of finite surjective morphisms $(\pi_i:Y_i\to X)_{i=1}^n$  with each $Y_i$ geometrically integral over $k$ and $\deg \pi_i \geq 2$, the set $X(k)\setminus \cup_{i=1}^n \pi_i(Y_i(k))$ is dense in $X$. In particular, if $X$ satisfies the Hilbert property over $k$, then $X(k)$ is dense. One can not expect the Hilbert property to hold for every variety with a dense set of rational points. Indeed,  a normal projective variety over a number field $k$ with the Hilbert property is (geometrically) simply-connected by Corvaja-Zannier's \cite[Theorem~1.6]{CZHP}. 

Corvaja-Zannier's result on the triviality of the geometric fundamental group of a variety over a number field with the Hilbert property can be rephrased as saying that one   can not expect a variety to have "more" rational points than any  finite collection of its non-trivial \'etale coverings. It is only reasonable to expect a variety to have more points than its \emph{ramified} covers. This leads one to the following notion of Corvaja-Zannier.  
  
  \begin{definition}[Corvaja-Zannier] Let $k$ be a field of characteristic zero. A   normal projective geometrically connected variety  $X$ over   $k$ satisfies the \emph{weak-Hilbert property   over $k$} if, for every  finite collection    $(\pi_i:Y_i\to X)_{i=1}^n$ of finite surjective ramified morphisms with each $Y_i$ a geometrically integral   normal variety over $k$, the set
  \[
  X(k) \setminus \cup_{i=1}^n \pi_i(Y_i(k))
  \] is dense in $X$.
  \end{definition}
  
  The weak-Hilbert property is   expected to hold for all varieties over a number field with a dense set of rational points \cite[Question-Conjecture~2]{CZHP}. This conjecture predicts,   roughly speaking, that a variety with a dense set of rational points has even more points than one sees at first hand. Namely, such a variety probably has more points than any of its ramified covers!

The following extension of Corvaja-Zannier's notion of the weak-Hilbert property will be useful. 
  
  \begin{definition}[Hilbertian pairs]\label{defn} 
  Let $X$ be a normal projective variety over a field  $k$ of characteristic zero and let $\Omega\subset X(k)$ be a dense subset. The pair  $(X,\Omega)$ has the \emph{weak-Hilbert property} (or: is \emph{a Hilbertian pair}) if, for every finite field extension $L/k$ and every finite collection    $(\pi_i:Y_i\to X_L)_{i=1}^n$ of finite surjective ramified morphisms with each $Y_{i}$ a geometrically integral   normal variety over $L$, the set 
  \[\Omega \setminus \cup_{i=1}^n \pi_i(Y_i(L))\] is dense in $X$.
  \end{definition}

Besides being useful in the proof of our main result below, the notion of Hilbertian pairs   also shows the ubiquity of the weak-Hilbert property in the study of Lang's conjecture \cite{JBook, Lang2}. Namely,    Lang's conjecture predicts that a ramified cover of an abelian variety over a number field does not have a dense set of rational points. This non-density statement is equivalent to the statement that,  for every abelian variety $A$ over a number field  $k$ and every dense subset $\Omega\subset A(k)$, the pair $(A,\Omega)$ has the weak-Hilbert property.

  Our main   result is a criterion for a variety admitting a suitable fibration to have the weak-Hilbert property, and reads as follows.

   \begin{theorem}[Mixed fibration theorem]\label{thm1} Let $k$ be a number  field,   let $S$ be a normal projective variety   over $k$,   and let $f:X\to S$ be a morphism of normal projective varieties over $k$.
 Assume that there is a dense   subset $\Omega\subset S(k)$ such that 
 \begin{enumerate}
 \item  for every $s$ in $\Omega$, the projective variety $X_s$ is an integral normal variety satisfying the Hilbert property over $k$;
 \item the pair $(S,\Omega)$ satisfies the weak-Hilbert property.
\end{enumerate}
  Then     $X$ satisfies the weak-Hilbert property over $k$.
  \end{theorem}

  We refer to Theorem \ref{thm1} as a "mixed" fibration theorem, as the base satisfies the weak-Hilbert property and some of its fibres satisfy the (usual) Hilbert property. (In particular, these fibres are  geometrically simply connected.) Note that Theorem \ref{thm1} is an analogue of  Bary-Soroker--Fehm--Petersen's fibration theorem in which the  base satisfies the Hilbert property (and is thus simply-connected); see \cite{BarySoroker}. Recall that their result can be used    to give an affirmative answer to an old question of Serre  \cite[\S3.1]{Serre} on products of varieties with the Hilbert property (Serre's question can also be answered using \cite[Lemma~8.12]{HarpazWittenberg}). The proof of Theorem \ref{thm1} is  inspired by arguments of Bary-Soroker--Fehm--Petersen and some arguments in the proof of \cite[Theorem~2.3]{Jell}. 
  
  We stress that the restriction to number fields is necessary in our proof, as we invoke the fact that the general fibres of $X\to S$ are geometrically simply-connected (as proven by Corvaja-Zannier).  
  
  Our "mixed" fibration theorem can be applied to certain fibrations over abelian varieties.
  
  \begin{theorem} \label{thm2} Let $k$ be a  number field.
Let $A$ be an abelian variety over $k$ and let $X\to A$ be a smooth proper geometrically connected morphism whose geometric fibres are   homogeneous spaces of  connected linear   groups.  Then there is a finite field extension $L/k$ such that  $X_L$ has the weak-Hilbert property over $L$.
  \end{theorem}

  Crucial to our proof of Theorem \ref{thm2} is the result of Colliot-Th\'el\`ene and Iyer that $X$ satisfies potential density of rational points \cite[Corollary~2.5]{CTI}, i.e., there is a finite field extension $L/k$ such that $X(L)$ is dense. We stress that, to prove Theorem \ref{thm2}, we  use the more precise version \cite[Theorem~2.3]{CTI} of their density result. Moreover, 
   our proof of Theorem \ref{thm2} also  crucially uses   the recently established analogue of    Hilbert's irreducibility theorem (i.e., the weak-Hilbert property) for abelian varieties  \cite[Theorem~1.3]{CDJLZ}.

Note that Theorem \ref{thm1} together with Hilbert's irreducibility theorem for rational varieties and abelian varieties (see \cite[\S 3]{Serre} and \cite{CDJLZ})   implies that, for $X$ a smooth proper geometrically connected variety over a number field $k$  such that $X_{\overline{k}}$ is birational to a semi-abelian variety, there is a number field $L/k$ such that $X_L$  has the weak-Hilbert property over $L$.
 
  \subsection{Motivation}
Potential density of rational points is conjecturally closely related to Campana's notion of special variety \cite{CampanaOr0}. In fact, Campana conjectured that a projective variety satisfies potential density if and only if it is special    (see     \cite[Conjecture~9.20]{CampanaOr0}).  Subsequently, Corvaja-Zannier conjectured that varieties with a dense set of rational points have the weak-Hilbert property \cite{CZHP}. We are interested in the following consequence of their conjectures.

   \begin{conjecture}[Campana, Corvaja-Zannier]\label{conj} If $X$   is a  special smooth projective geometrically connected variety over a number field $k$, then   there is a  number field $L/k$ such that $X_L$ satisfies    the  weak-Hilbert property over $L$. 
  \end{conjecture}
   
   This conjecture holds for rational varieties 
  \cite[Theorem~3.4.1]{Serre}  (i.e., $T_X$ ample) and abelian varieties   \cite{CDJLZ}  (i.e., $T_X$ trivial); see  \cite{CTS, CZHP, Demeio2,    GvirtzChenMezzedimi, NakaharaStreeter, Streeter, ZannierDuke} for more examples.

    Our main result (Theorem \ref{thm2}) provides a new case of Conjecture \ref{conj} which interpolates between the case of ample tangent bundle and trivial tangent bundle.  Let us explain in what sense Theorem \ref{thm2} settles the above conjecture for all varieties with nef tangent bundle.
    
  Firstly, a smooth projective variety with a nef tangent bundle is special and should therefore satisfy the weak-Hilbert property over some large enough number field (Conjecture \ref{conj}). Moreover,     in \cite{DemaillyPeternellSchneider}   Demailly, Peternell, and Schneider proved that,  for $k$ an algebraically closed field of characteristic zero and  $X$  a smooth projective variety over $k$ with nef tangent bundle,   there is a finite \'etale cover $Y\to X$ with $Y$ connected such that the Albanese map $Y\to\mathrm{Alb}(Y)$ of $Y$ is smooth surjective and its geometric fibres are smooth projective Fano varieties with nef tangent bundle. 
The expectation is that varieties with nef tangent bundle are actually built out of  homogeneous spaces  over an abelian variety, and this is   made more precise by the following conjecture of Campana--Peternell \cite{CampanaPeternell} (which is concerned with the fibres of the Albanese map of $Y$).
  
  \begin{conjecture}[Campana--Peternell] Let $k$ be an algebraically closed field of characteristic zero.
 Let $X$ be a smooth projective Fano variety over $k$  with nef tangent bundle.    Then $X$ is a rational (smooth projective) homogeneous space, i.e.,  a homogeneous space of a connected linear group.
  \end{conjecture}
  
  
  We use that this conjecture holds in dimension at most five to obtain the following result.
  \begin{theorem}\label{thm:final}
Let $X$ be a smooth projective geometrically connected variety over a number field $k$ such that the tangent bundle $T_X$ is nef.  
  If $\dim X \leq 5$, then there is a finite field extension $L/k$ such that $X_L$ has the weak-Hilbert property over $L$.
  \end{theorem}
    
    Similarly, we have the following:
    
    \begin{theorem}\label{thm:final2} Assume the Campana-Peternell conjecture holds.
Let $X$ be a smooth projective geometrically connected variety over a number field $k$ such that the tangent bundle $T_X$ is nef.   Then there is a finite field extension $L/k$ such that $X_L$ has the weak-Hilbert property over $L$.
    \end{theorem}

 \begin{ack}  We  gratefully acknowledge Olivier Wittenberg for very helpful discussions on the work of Colliot-Th\'el\`ene--Iyer \cite{CTI}, and also on how to prove the mixed fibration theorem. We gratefully acknowledge the IHES.
  \end{ack}
 
 \begin{con}
 If $k$ is a field, then a variety over $k$ is a finite type separated   scheme over $k$.   
 A morphism of schemes $f:X\to Y$ is \emph{ramified} if it is not unramified. If $X\to S$ is a morphism of schemes and $s\in S$, then $X_s$ denotes the scheme-theoretic fibre  of $X$ over $s$.
 \end{con}

  \begin{proof}   [Proof of Theorem \ref{thm1}]    
To show that $X$ has the weak-Hilbert property, let $n\geq 1$ be an integer and let $(\pi_i:Y_i\to X)_{i=1}^n$ be a finite collection of finite surjective ramified morphisms with $Y_i$ a geometrically integral normal variety over $k$.  
For each $i$, we consider the composed morphism $f_i:Y_i\to X\to S$, where $f:X\to S$ is as in the statement of the theorem. For each $i$, let $Y_i\to T_i\to S$ be the Stein factorization of $f_i$, where $Y_i\to T_i$ has geometrically connected fibres,  $T_i$ is normal (and geometrically integral), and $T_i\to S$ is a finite surjective morphism.   

 To prove the theorem, we may and do assume that there is an integer $1< r < n$ such that the following holds.
\begin{enumerate}
\item For every $i$ with $1\leq i \leq r$, the morphism $\psi_i:T_i\to S$ is unramified, and
\item for every $i$ with $r< i \leq n$, the morphism $\psi_i:T_i\to S$ is ramified.
\end{enumerate} Indeed,  this can be achieved by   enlarging $n$ and adding to $(\pi_i:Y_i\to X)_{i=1}^n$ the trivial cover and the pull-back of some ramified cover of $S$ (if necessary).

Recall that a scheme is normal if all its local rings are normal integral domains. (With this definition, a normal scheme could be disconnected.)
Let $U\subset S$ be a  dense open subscheme such that, for every $u$ in $U$, the scheme $Y_u$ is   normal over $k(u)$ and the scheme $X_u$ is a geometrically integral and geometrically normal variety over $k(u)$; such a dense open subscheme exists by spreading out techniques (see \cite[Théorème~6.9.1]{EGAIVII} and \cite[Théorème 9.7.7 (iv) and Théorème~12.2.4 (iv)]{EGAIVIII}).

 Define $$\Sigma:=\Omega\cap U \setminus \bigcup_{i=r+1}^n \psi_i(T_i(k))\subset S(k).$$  Moreover, define 
\[
\Psi := \bigcup_{s\in \Sigma} X_s(k) \setminus \cup_{i=1}^{r} \pi_{i,s}(Y_{i,s}(k))
\] Note that 
\[
\Psi \subset X(k) \setminus \cup_{i=1}^n \pi_i(Y_i(k)).
\] Thus,  to conclude the proof, it suffices to show that $\Psi$ is dense in $X$. 

First, 
since $(S,\Omega)$ satisfies the weak-Hilbert property over $k$ (by assumption), the set $$\Sigma=\Omega\cap U \setminus \bigcup_{i=r+1}^n \psi_i(T_i(k))\subset S(k)$$ is dense in $S$. Thus, to prove that $\Psi$ is dense, it suffices to show that,  for every $s$ in $\Sigma$, the set 
\[X_s(k)\setminus \cup_{i=1}^r \pi_{i,s}(Y_{i,s}(k)) \] is dense in $X_s$. (Here we will use that $X_s$  satisfies the Hilbert property over $k$.)

For $1\leq i \leq r$ and $s\in \Omega$, let $t_1,\ldots, t_{N_i}$ be the points of the fibre $\psi_i^{-1}(s)$. Note that $Y_{i,s} = \sqcup_{j=1}^{N_i} Y_{i,t_j}$ and that each $Y_{i,t_j}$ is geometrically connected, and thus geometrically integral and geometrically normal, over $k$.  We claim that,  for every $1\leq j \leq N_i$, the finite morphism $Y_{i,t_j}\to X_s$ is of degree $>1$.   Since $X_s$ has the Hilbert property, this then implies that 
\[X_s(k)\setminus \cup_{i=1}^r \pi_{i,s}(Y_{i,s}(k)) \] is dense in $X_s$ (as required).
 
 To prove that $Y_{i,t_j}\to X_s$ is of degree $>1$, first note that $T_i\to S$ is finite \'etale for every $1\leq i\leq r$.  In particular, the morphism $X_{T_i}\to X$ with $X_{T_i} := X\times_S T$ is finite \'etale. Thus, the induced morphism $Y_i\to X_{T_i}$ is ramified (otherwise the morphism $Y_i\to X$ would be \'etale, contradicting our assumption that it is ramified.) Since $Y_i\to X_{T_i}$ is ramified and the branch locus of $Y_i\to X_T$ dominates $T$ (because the fibres of $Y_i\to T$ are simply connected \cite[Theorem~1.6]{CZHP}),  for every $1\leq j \leq N_i$, the morphism $Y_{i,t_j}\to X_t$ is ramified (hence of degree $>1$).   
 This concludes the proof.
 \end{proof}

 \begin{proof}[Proof of Theorem \ref{thm2}] 
 Replacing $k$ by a larger number field if necessary, we may  and do assume that $\Omega:=A(k)$ is dense;  here we use that abelian  varieties satisfy potential density  \cite[\S 3]{HassettTschinkel}. By  \cite[Theorem~2.3]{CTI}, there is a number field $L/k$ such that $\Omega:=A(k)$ (as a subset of $A(L)$) lies in the image of $X(L) \to A(L)$.  Define $S:=A_L$ and $\mathcal{X}:= X_L$ and consider $\mathcal{X}\to S$. By  \cite[Theorem~1.3]{CDJLZ}, the pair  $(S,\Omega)$ has the weak-Hilbert property. Moreover,  for every $s$ in $\Omega$, the fibre $X_s$ has an $L$-point. Since $X_s$ is a geometrically-rational (smooth projective) homogeneous space over $L$ with an $L$-point,  it is rational over $L$ (use \cite[Proposition~1.3]{CTI}). In particular, since the Hilbert property is a birational invariant, we see that    $X_s$ has the Hilbert property over $L$ (use \cite[Theorem~3.4.1]{Serre}).  By Theorem \ref{thm1}, we conclude that $X_L = \mathcal{X}$  has the weak-Hilbert property over $L$,   as $X_L\to S$ and $\Omega$ satisfy the conditions of the mixed fibration theorem (Theorem \ref{thm1}).
 \end{proof}

\begin{proof}[Proof of Theorem \ref{thm:final}]  Let $\overline{k}$ be an algebraic closure of $k$.   By \cite{DemaillyPeternellSchneider}, there is a smooth projective connected variety $Y$ over $\overline{k}$ and a finite \'etale morphism $Y\to X_{\overline{k}}$ such that the Albanese map $Y\to \mathrm{Alb}(Y)$ (with respect to some choice of basepoint $y$ in $Y(\overline{k})$) is smooth surjective
and its fibres are Fano varieties with nef tangent bundle.  Since $\dim X\leq 5$, it follows that $\dim Y\leq 5$, so that the fibres of $Y\to \mathrm{Alb}(Y)$ have dimension at most $5$. In particular,  the conjecture of Campana-Peternell holds for every fibre of $Y\to\mathrm{Alb}(Y)$ (see   \cite{Kanemitsu2,  Watanabe}), so that the fibres of $Y\to \mathrm{Alb}(Y)$ are homogeneous spaces under a connected linear group. 

We now descend all of the above to a finite extension of $k$. More precisely,  replacing $k$ by a larger number field if necessary,  it follows from the preceding paragraph that there is a smooth projective geometrically connected variety $Z$ over $k$ and a finite \'etale morphism $Z\to X$  such that the Albanese map $Z\to \mathrm{Alb}(Z)$ is smooth surjective and its fibres are    homogeneous spaces under a connected linear group. In particular, by Theorem \ref{thm2}, there is a finite field extension $L/k$ such that $Z_L$ has the weak-Hilbert property over $L$.   Since  the weak-Hilbert property behaves well with respect to finite \'etale quotients \cite[Theorem~3.7]{CDJLZ} and the morphism $Z_L\to X_L$ is finite \'etale surjective, it follows that $X_L$ has the weak-Hilbert property over $L$.
\end{proof}

  \begin{proof}[Proof of Theorem \ref{thm:final2}]
 This is similar to the proof of  Theorem \ref{thm:final}. Namely,   replacing $k$  by a finite extension if necessary,   by \cite{DemaillyPeternellSchneider} and Campana-Peternell's conjecture, there is a smooth projective geometrically connected variety $Z$ over $k$ and a finite \'etale morphism $Z\to X$  such that the Albanese map $Z\to \mathrm{Alb}(Z)$ is smooth surjective and its fibres are    homogeneous spaces under a connected linear group. In particular, by Theorem \ref{thm2}, there is a finite field extension $L/k$ such that $Z_L$ has the weak-Hilbert property over $L$.   Since  the weak-Hilbert property behaves well with respect to finite \'etale quotients \cite[Theorem~3.7]{CDJLZ} and the morphism $Z_L\to X_L$ is finite \'etale surjective, it follows that $X_L$ has the weak-Hilbert property over $L$.
  \end{proof}

 \bibliography{refsci}{}
\bibliographystyle{alpha}

\end{document}